\documentclass[12pt]{amsart}
\pdfoutput=1

\usepackage{cite}
\usepackage{lineno,hyperref}
\modulolinenumbers[5]
\usepackage{epstopdf}
\usepackage{pifont}
\usepackage{latexsym}
\usepackage{graphics}
\usepackage{graphicx}
\usepackage{float}
\usepackage[dvips]{xy}
\xyoption{all}
\usepackage{ifpdf}
\usepackage{multirow}
\usepackage{amssymb, amsthm, amsmath, mathrsfs, stmaryrd }
\usepackage{dsfont}
\usepackage{hhline}
\usepackage{centernot}
\usepackage{verbatim, comment, color}
\usepackage{enumerate}
\usepackage{setspace}
\usepackage{lipsum}% <- For dummy text
\usepackage{subcaption}

\usepackage{mathtools}

\newtheorem{theorem}{Theorem} %[section]
\newtheorem{proposition}[theorem]{Proposition}

\newcommand{\ba}{\begin{align}}
\newcommand{\ea}{\end{align}}  %WTH is wrong with this??????

\newcommand{\be}{\begin{equation}}

\newcommand{\ee}{\end{equation}}
\newcommand{\bea}{\begin{eqnarray}}
\newcommand{\eea}{\end{eqnarray}}
\newcommand{\barr}{\begin{array}}
\newcommand{\earr}{\end{array}}
\newcommand{\bn}{\begin{enumerate}}
\newcommand{\en}{\end{enumerate}}
\newcommand{\bi}{\begin{itemize}}
\newcommand{\ei}{\end{itemize}}
\newcommand{\bbbm}{\begin{pmatrix}}
\newcommand{\eeem}{\end{pmatrix}}

\newcommand{\cN}{{\cal N}}
\newcommand{\cP}{{\cal P}}

\newcommand{\cX}{{\cal X}}

\newcommand{\cM}{{\cal M}}

\newcommand{\R}{{\mathbb R}}
\newcommand{\N}{{\mathbb N}}
\newcommand{\E}{\mathbb{E}}

\newcommand{\ga}{\gamma}

\newcommand{\la}{\lambda}

\newcommand{\tta}{\theta}
\newcommand{\ignore}[1]{}{}

\newcommand{\nn}{\nonumber}

\newcommand{\q}{\quad}

\newcommand{{\QED}}{{\hfill QED} \smallskip}

\newcommand{\spt}{\mathop{\rm spt}}

\renewcommand{\subset}{\subseteq}

\renewcommand{\phi}{\varphi}
\newcommand{\cal}{\mathcal}

 \definecolor{darkspringgreen}{rgb}{0.09, 0.45, 0.27} 
 \definecolor{darkgray}{rgb}{0.66, 0.66, 0.66}
\numberwithin{equation}{section}
\numberwithin{theorem}{section}

\begin{document}
\title
[Optimal American option exercise under model uncertainty]
{
Optimal exercise decision of American options under model uncertainty
%Structure of optimal American option exercise under model uncertainty
%Structure of optimal stopping of American options under model uncertainty
} 

\thanks{
Version 1 of the paper posted on arXiv had an incorrect Proposition 2.1, which was used to erroneously derive the equation $P_c = \overline P_c$. The proposition was removed in Ver 2, and the main theorem now assumes the equation. We would like to find sufficient conditions for the equation.
}
\date{\today}

\author{Tongseok Lim}
\address{Tongseok Lim: Mitchell E. Daniels, Jr. School of Business \newline  Purdue University, West Lafayette, Indiana 47907, USA}
\email{lim336@purdue.edu}
\onehalfspacing

\begin{abstract}
Given the marginal distribution information of the underlying asset price at two future times $T_1$ and $T_2$, we consider the problem of determining a model-free upper bound on the price of a class of American options that must be exercised at either $T_1$ or $T_2$. The model uncertainty consistent with the given marginal information is described as the martingale optimal transport problem. We show that any option exercise scheme associated with any market model that jointly maximizes the expected option payoff must be nonrandomized if the American option payoff satisfies a suitable convexity condition and the model-free price upper bound and its relaxed version coincide. The latter condition is desired to be removed under appropriate conditions on the cost and marginals.
\end{abstract}

\maketitle


\section{Introduction}
This paper was mainly inspired by Hobson and Norgilas \cite{HoNo19}, Aksamit, Deng, Obłój and Tan \cite{adot19}, as well as Beiglb{\"o}ck and Juillet \cite{bj} and  Beiglb{\"o}ck, Nutz and Touzi \cite{bnt}. A related problem in continuous time setup was studied in Bayraktar, Cox and Stoev \cite{bcs18}. We consider two future times $0 < T_1 < T_2$ and an asset price process $(X,Y)$, where $X, Y$ represents the asset price at time $T_1, T_2$, respectively. Let $\cP(\cX)$ denote the set of all probability measures/distributions over a set $\cX$ with finite first moment. Let $\mu, \nu \in \cP(\R)$ be probability measures in convex order:
\begin{align*}\label{convexorder}
\mu \preceq_c \nu \ \ \text{if} \  \ \mu(f) \le \nu(f) \ \ \text{for every convex function } f \text{ on } \R,
\end{align*}
where $\mu(f) := \E_\mu[f(X)] = \int f(x) \mu (dx)$. We consider market models that are defined by the following set of martingale transports from $\mu$ to $\nu$:\begin{align}
\cM(\mu,\nu) = \{ \pi \in \cP(\R^{2}) \, | \, \pi = {\rm Law} (X,Y), \E_\pi[Y | X]=X, {\rm Law}(X) = \mu, {\rm Law}(Y) = \nu \}. \nn
\end{align}
In finance, each $\pi \in \cM(\mu,\nu)$ represents a feasible joint law of the price $(X, Y)$ given the marginal information $\mu, \nu$ in the (two-period) market, under which $(X, Y)$ is a martingale, written as $\E_\pi [Y | X ] = X$. It is well known that the condition $\mu \preceq_c \nu$ is equivalent to $\cM(\mu,\nu) \neq \emptyset$. We refer to \cite{CKPS21, H17, GaHeTo11, D16} for further background.

We consider the cost function which describes an American option payoff 
\be\label{cost}
c = (c_1, c_2) = (c_1(x), c_2(x,y)), \q c_1,c_2 \in \R,
\ee
such that if an obligee (option holder) selects $c_1$, she receives the payout $c_1(X)$, otherwise she receives the payout $c_2(X,Y)$. Thus, in the former case, her payout is determined at time 1, whereas it is determined at time 2 in the latter. We assume she can make this choice conditional on the price $X=x$, and that she can also randomize (or split) her choice, represented by a Borel function $s: \R \to [0,1]$. This means that given $X=x$, she exercises $c_1$ with probability (or proportion) $s(x)$, otherwise $c_2$ with probability $1- s(x)$. Given a function $s : \R \to \R$ and a measure $\mu$ on $\R$, let the measure $s \mu$  be given by $s\mu (B) = \int_B s(x) \mu(dx)$. Since $\mu$ is fixed, the choice of a randomization $s$ is equivalent to the choice of $0 \le \mu_1 \le \mu$,\footnote{All measures/distributions in this paper are assumed to be non-negative.} such that with $\mu_2 := \mu - \mu_1$, $s_1:=s$, $s_2:= 1-s$ equals the Radon–Nikodym derivative $\frac{d\mu_1}{d\mu}, \frac{d\mu_2}{d\mu}$ $\mu$-a.s., respectively. This leads us to consider the optimization problem
\begin{align}\label{problem}
P_c := \sup_{\pi \in \cM(\mu,\nu)} \sup_{\mu_1 \le \mu}  \ \ \E_{\ga_1} [c_1] + \E_{\ga_2} [c_2],
\end{align}
where for a given $\pi = \pi_x \otimes \mu \in \cM(\mu,\nu)$,\footnote{Any $\pi = {\rm Law}(X,Y) \in \cP(\R^2)$, representing the joint law of the random variables $X$ and $Y$, can be written as $\pi= \pi_x \otimes {\rm Law}(X)$, where $\pi_x \in \cP(\R)$ is called a kernel of $\pi$ with respect to ${\rm Law}(X)$.  $\pi_x$ represents the conditional distribution of $Y$ given $X=x$, i.e.,  $\pi_x(B) = \cP(Y \in B \, | \, X=x )$ for all Borel set $B \subset \R$. Note that $\pi = \pi_x \otimes \mu \in \cM(\mu,\nu)$ iff $\int y\, \pi_x(dy) = x$ $\mu$-a.e.\,$x$.}
 we define $\ga_l = \pi_x \otimes \mu_l$, $l=1,2$, such that $\ga_1 + \ga_2 = \pi$ and that  $\ga_1$ and $\ga_2$ share the same kernel $\{\pi_x\}_x$ inherited from $\pi$.
 
 In view of the obligor (the person responsible for the payment of the option), a solution $(\pi, \mu_1)$ to \eqref{problem} represents a worst-possible market scenario $\pi$ combined with the option exercise scheme $\mu_1$, yielding the maximum expected payout $P_c$. %In this viewpoint, $\frac{d\mu_1}{d\mu}, \frac{d\mu_2}{d\mu}$ represents the obligor's expectation of how the obligee will exercise the option. 
 
We will assume the following regularity condition on $c$ throughout the paper. 

\noindent{\bf [A]} Throughout the paper, we assume that $c_1, c_2$ are continuous, $\mu \preceq_c \nu$, and that the marginals $\mu, \nu$ satisfy the following condition: there exist continuous functions $v \in L^1(\mu)$, $w \in L^1(\nu)$  such that $|c_1| + |c_2| \le v(x) + w(y) $. Note that this implies
\be
\big| \sum_l  \E_{\ga_l} [c_l] \big| \le  \sum_l  \E_{\ga_l} [| c_l |] \le  \sum_l  \E_{\pi} [| c_l |] \le \mu(v) + \nu(w)  < \infty \, \text{ for any } \pi \in \cM(\mu, \nu). \nn
\ee
This in turn implies that the problem \eqref{problem} is attained (i.e., admits an optimizer) by a standard argument in the calculus of variations \cite{Sa15}.

 \noindent\cite{HoNo19} considered a specific cost called an American put, whose payoff is given by
\be\label{Hobsoncost}
c_1(x) = (K_1 - x)^+,\q c_2(x,y) = c_2(y) = (K_2 - y)^+, \q K_1 > K_2,
\ee
and considered those option exercise schemes which are {\em pure}, or {\em non-randomized}; that is, \cite{HoNo19}  assumed that the obligee can only choose a Borel set $B \subset \R$ in which she selects $c_1$ if $x \in B$ and $c_2$ otherwise. In terms of $\mu_1$, notice that this is equivalent to the statement that $\mu_1$ and $ \mu_2$ are mutually disjoint, written as $\mu_1 \perp \mu_2$ (while $\mu_1 + \mu_2 = \mu$). In other words, \cite{HoNo19} assumed that $\mu_1, \mu_2$ must saturate $\mu$ on their respective supports. In addition, \cite{HoNo19} assumed that $\mu$ is continuous, i.e., has no atoms. Under these assumptions, \cite{HoNo19} showed that an optimal market model $\pi$ for the problem \eqref{problem} is given by the {\em left-curtain coupling} (see \cite{bj, HoNo19, HT16} for more details about this interesting martingale transport) along with an optimal exercise strategy $B$, and furthermore, the cheapest superhedge can be derived.

Now we would like to shift our focus and ask, ``Under what conditions must the optimal option exercise be pure?" That is, when will an optimal $\mu_1$ saturate $\mu$, or equivalently, achieve $\mu_1 \perp \mu_2$? Note that the problem \eqref{problem} can be rewritten as 
\begin{align}\label{problem2}
P_c =\sup_{\mu_1 \le \mu} P_c(\mu_1), \ \text{ where } \ P_c(\mu_1)  := \sup_{\pi \in \cM(\mu,\nu)} \ \E_{\ga_1} [c_1] + \E_{\ga_2} [c_2],
\end{align}
where $\ga_l = \pi_x \otimes \mu_l$, $l=1,2$. Note that the problem \eqref{problem} has a nonconvex domain in terms of the variable $(\ga_1,\ga_2)$. This is because even if  $(\ga_1, \ga_2)$, $(\ga_1', \ga_2')$ are feasible (i.e., sharing the same kernel respectively), the convex combination $(\frac{\ga_1 + \ga_1'}{2},  \frac{\ga_2 + \ga_2'}{2})$ may not share the same kernel thus infeasible, unless $\mu_1 = \mu_1'$ and $\mu_2 = \mu_2'$. On the other hand, the subproblem $P_c(\mu_1)$ has a convex domain in terms of $(\ga_1, \ga_2)$. This leads us to consider a relaxed problem \eqref{generalrelaxedproblem} with its optimal value denoted by $\overline P_c$. Clearly $P_c \le \overline P_c$; see Section \ref{dualsection} for details. Our result is the following.

\begin{theorem}\label{main}
Assume {\bf [A]} and the cost form \eqref{cost}. Suppose $y \mapsto c_2(x,y)$ is strictly convex and $c_1(x) \neq c_2(x,x)$ for $\mu$-a.e.\,$x$, and  $\nu$ is absolutely continuous with respect to the Lebesgue measure. If $P_c = \overline P_c$, then every solution $(\pi, \mu_1)$ to the problem \eqref{problem2}  satisfies $\mu_1 \perp \mu - \mu_1$. Furthermore, given any optimal candidate model $\pi$, the $\mu_1$ yielding an optimal pair $(\pi, \mu_1)$ is unique.
\end{theorem}
We note that the condition $c_1(x) > c_2(x,x)$ is natural because, if $c_1(x) \le c_2(x,x)$ and $c_2$ is convex in $y$, it is always optimal to choose $c_2(x,y)$ by Jensen's inequality $c_2(x,x) \le \int c_2(x,y) \pi_x(dy)$. Theorem \ref{main} says that in this case, every optimal exercise, or stopping, is nonrandomized. Evidently, the problem \eqref{problem} can be viewed as an optimal stopping problem, in which the option holder either stops at time 1 and receives the sure reward $c_1(x)$, or goes and receives the reward $c_2(x,y)$ (which is stochastic at time 1) at time 2. This naturally places the theorem in the context of the vast literature on the Skorokhod embedding problem \cite{bch, Ho11, Obloj}, with the key difference that we now face uncertainty in the family of models $\cM(\mu, \nu)$. Such model uncertainty was also considered in \cite{bcs18, ds1} in continuous time setup. For more results on American options and their robust hedging, we refer to \cite{bz16, bz17, bz19}.  

In the optimal transport literature, the absolute continuity of $\mu$ is typically assumed in order to derive non-randomizing solutions, known as Monge solutions. Continuity of $\mu$ was also assumed in \cite{HoNo19}. In contrast, Theorem \ref{main} assumes the absolute continuity of $\nu$, while making no assumptions about $\mu$. On the other hand, the equation assumption $P_c = \overline P_c$ imposed in the theorem appears to be highly  restrictive, prompting us to seek a sufficient condition that yields the equation. For example, can the absolute continuity of $\mu$ with respect to Lebesgue measure imply the equation (with suitable additional conditions on the cost)?

Finally, the uniqueness of $\mu_1$ given a fixed model $\pi$ is obtained by a standard argument in optimal transport through mixing two optimal solutions and invoking the result $\mu_1 \perp \mu - \mu_1$. When $(\pi, \mu_1)$ and $(\pi', \mu_1')$ are both optimal (with possibly $\pi \ne \pi'$), it is an open question whether $\mu_1 = \mu_1'$ under suitable conditions. This is due to the nonconvexity of the domain of the problem \eqref{problem} in terms of $(\ga_1,\ga_2)$.

The remainder of the paper is structured as follows. The theorem will be proved utilizing a duality and its attainment result. They will be discussed in Section \ref{dualsection}. Section \ref{proofs} then presents proofs of the results.

\section{Duality}\label{dualsection}
In this section, we consider cost functions more general than \eqref{cost}, such as
\be\label{cost2}
\vec c = (c_1, c_2,...,c_L), \ c_l = c_l(x,y) \in \R, \ l=1,2,...,L.
\ee
Throughout this section, we assume the following.\\
 {\bf [A']}  $c_l$ are continuous for all $l$,  $\mu \preceq_c \nu$, and $\sum_{l=1}^L |c_l(x,y)| \le v(x) + w(y) $ for some continuous functions $v \in L^1(\mu)$, $w \in L^1(\nu)$. 

 As noted, the domain of the problem \eqref{problem}, in terms of the variable $(\ga_1,\ga_2)$, is nonconvex. This leads us to consider a relaxed problem for \eqref{problem}; see also \cite{adot19} for related results. Let $\cM := \cup_{\mu \preceq_c \nu} \cM(\mu, \nu)$, that is, $\cM$ is the set of all martingale transports between some probability marginals in convex order, hence $\cM \subset \cP(\R^2)$. Let $\overline \cM$ be the set of all martingale transports with arbitrary nonnegative finite total mass, that is, $\ga \in \overline \cM$ if $\ga \equiv 0$ or $\ga / ||\ga|| \in \cM$ where $|| \ga || = \int_{\R^2} \ga(dx,dy) \in (0, \infty)$ denotes the total mass. Define
\be
\cM_L(\mu, \nu) := \bigg\{\vec \ga = (\ga_1,...,\ga_L) \, \bigg| \, \sum_{l=1}^L \ga_l \in \cM(\mu,\nu) \text{ and } \ga_l \in  \overline \cM \text { for all } l=1,...,L. \bigg\}\nn
\ee
$\cM_L(\mu, \nu)$ is clearly convex. Now we define the relaxed problem 
\begin{align}\label{generalrelaxedproblem}
\overline P_c  := \sup_{\vec \ga \in \cM_L(\mu,\nu)} \sum_{l=1}^L \E_{\ga_l} [c_l]. \end{align}
The difference is that in \eqref{problem} (with the generalized cost \eqref{cost2}), $\{ \ga_l\}_l$ are assumed to have the same kernel $\pi_x$ inherited from a model $\pi \in \cM(\mu, \nu)$, whereas in \eqref{generalrelaxedproblem}, this restriction is relaxed. Both problems satisfy the condition $\sum_l \ga_l \in \cM(\mu,\nu)$. Hence, $P_c \le \overline P_c$. 

We turn to  the dual problem of \eqref{generalrelaxedproblem}. Define $\overline \Psi_c$ to be the space of functions $( \phi, \psi, \vec \tta) = ( \phi, \psi, \tta_1,...,\tta_L)$ such that $\phi \in C(\R) \cap L^1(\mu)$,  $\psi \in C(\R) \cap L^1(\nu)$,  $\tta_l \in C_b(\R)$, satisfying
\begin{align}
c_l(x,y) \le \phi(x) + \psi (y) + \tta_l(x) (y-x) \ \text{ for all } l=1,...,L \text{ and } (x,y) \in \R^2. \label{relaxedsubdualineq}
\end{align}
The dual problem to \eqref{generalrelaxedproblem} is now given by
\be\label{generalrelaxeddualproblem}
\overline D_c := \inf_{(\phi, \psi, \vec \tta) \in \overline\Psi_c }  \mu(\phi) +  \nu(\psi).
\ee
A duality result is the following.

\begin{proposition}\label{duality2}
Assume {\bf [A']}. Then $\overline P_c =  \overline D_c$.
\end{proposition}
For the financial meaning of the dual problems in terms of American option superhedging, we refer to \cite{adot19, bhz15, bz16, bz17, bz19, HoNe17, HoNo19, N07}. The additional element required to prove Theorem \ref{main} is the dual attainment result, which asserts that there is an appropriate solution to the dual problem \eqref{generalrelaxeddualproblem}. For $\xi \in \cP(\R)$, its potential function is defined by $u_\xi (x) := \int |x-y| d\xi(y)$. Then we say that a pair of probabilities $(\mu , \nu)$ in convex order is irreducible if the set $I:=\{ x \in \R \, | \, u_{\mu}(x) < u_{\nu}(x) \}$ is a connected (open) interval containing the full mass of $\mu$, i.e., $\mu(I)=\mu(\R)$. 

\begin{proposition}\label{dualattainment}
Assume {\bf [A']} and suppose $(\mu, \nu)$ is irreducible. Then there exists a {\em dual optimizer} $( \phi, \psi, \vec \tta)$, $\phi, \psi : \R \to \R \cup \{+\infty\}$, $\tta_l : \R \to \R$, that satisfies \eqref{relaxedsubdualineq} tightly in the following pathwise sense (but needs not be in $\overline \Psi_c$): 
\begin{align}
c_l(x,y) = \phi(x) + \psi (y) + \tta_l(x) (y-x) \ \ \ga_l - a.e., \ \text{ for all }\, l=1,...,L \label{relaxedsubdualeq}
\end{align}
for every solution $\vec \ga = (\ga_1,...,\ga_L)$ to the problem \eqref{generalrelaxedproblem}.
\end{proposition}
We emphasize that $(\phi, \psi, \vec \tta)$ may not be in $\overline \Psi_c$ but are only measurable, with $\phi, \psi$ real-valued $\mu, \nu$-a.s., respectively. They need not be integrable nor continuous. 

\section{Proofs}\label{proofs}

\begin{proof}[Proof of Proposition \ref{duality2}] Let $\cN$ be the set of all nonnegative finite measures on $\R^2$ (that do not need to be martingales.) For $\ga \in \cN$, let $\ga^X, \ga^Y$ denote its marginal on the $x,y$-coordinate respectively. Let $\phi \in C(\R) \cap L^1(\mu)$,  $\psi \in C(\R) \cap L^1(\nu)$,  $\tta_l \in C_b(\R)$. We assert that the following equalities hold:
\begin{align*}
&\overline P_c  = \sup_{\vec \ga \in \cM_L(\mu,\nu)} \sum_{l=1}^L \E_{\ga_l} [c_l]\\
&=   \sup_{\ga_l \in \cN\, \forall l} \inf_{(\phi, \psi, \vec \tta)} \textstyle\sum_l \ga_l (c_l) + (\mu - \sum_l \ga_l^X )(\phi) + (\nu - \sum_l \ga_l^Y )(\psi) - \sum_l \ga_l ( \tta_l(x) (y-x) )\\
&=    \inf_{(\phi, \psi, \vec \tta)} \sup_{\ga_l \in \cN\, \forall l} \textstyle
\mu( \phi) + \nu(\psi) + \sum_l \ga_l ( c_l(x,y) - \phi(x) - \psi(y) - \tta_l(x)(y-x) ) \\
&= \inf_{c_l(x,y) \le  \phi(x) + \psi(y) +  \tta_l(x)(y-x)\, \forall l }   \mu( \phi) +\nu(\psi) = \overline D_c.
  \end{align*}
The derivation of the equalities is fairly standard: the second equality holds because the infimum achieves $-\infty$ as soon as $\sum_l \ga_l^X \ne \mu$, $\sum_l \ga_l^Y \ne \nu$, or $\ga_l \notin \overline \cM$, implying that $\vec \ga$ in the second line must be in $\cM_L(\mu,\nu)$ to achieve the first supremum. The third equality is based on a standard minimax theorem, which asserts that the equality holds when the sup and inf are swapped. Because the objective function is bilinear, i.e., linear in each variable ($(\ga_l)_l$ and $( \phi, \psi, \vec \tta)$), the minimax theorem holds in this case and we omit the detail. The fourth equality is because, if $c_l(x,y) - \phi(x) - \psi(y) - \tta_l(x)(y-x) >0$ for some $(x,y) \in \R^2$, one can select $\ga_l \in \cN$ such that the last supremum in the third line achieves $+\infty$, which hinders to achieve the first infimum. This implies $c_l(x,y) - \phi(x) - \psi(y) - \tta_l(x)(y-x) \le 0$ for all $(x,y)$, in which case it is best to choose $\ga_l \equiv 0$ for the supremum in the third line.
\end{proof}

\begin{proof}[Proof of Proposition \ref{dualattainment}]
The proof consists of extending the ideas in \cite{bj, bnt} to the vectorial cost \eqref{cost2}. We will follow the five steps illustrated in \cite{Lim23}, thereby omitting some details here but referring to the corresponding steps in \cite{Lim23}.

{\bf Step 1.} $\sum_{l=1}^L |c_l(x,y)| \le v(x) + w(y) $ for some continuous functions $v \in L^1(\mu)$, $w \in L^1(\nu)$. A dual optimizer exists for $\vec c$ iff so does for $\tilde c := ( c_l(x,y)+v(x) + w(y))_l$. Thus by replacing $\vec c = (c_1,...,c_L)$ with $\tilde c$, from now on we assume $c_l \ge 0$ for all $l$.

As $\overline P_c = \overline D_c \in \R$, we can find an approximating dual optimizer $(\phi_n, \psi_n, \tta_{l,n}) \in \overline \Psi_c$, ${n \in \N}$,  such that the following duality holds (for all $l=1,...,L$):
\begin{align}
\label{dual0} & \phi_{n} (x) + \psi_n(y) + \tta_{l,n}(x)(y-x) \ge c_l(x,y) \ge 0,\\
\label{maximizing} &\mu(\phi_n) + \nu(\psi_n) \searrow \overline P_c
\ \ \text{as} \ \ n \to \infty.
\end{align}
Define $f_n = - \phi_n$, $h_{l,n} = -\tta_{l,n}$, so that \eqref{dual0} becomes
\begin{align}
\label{dual}  f_{n} (x) +  h_{l,n}(x)(y-x) \le \psi_n(y) - c_l(x,y) \le \psi_n(y).
\end{align}
Define the convex functions
\begin{align}
\label{chidefinition}
\chi_{l,n}(y):= \sup_{x \in \R} f_{n} (x) +  h_{l,n}(x)(y-x), \q \chi_n := \sup_{l=1,...,L} \chi_{l,n}. 
\end{align}
Notice $\chi_{l,n}(y) \ge f_n(y) + h_{l,n}(y)(y-y) = f_n(y)$ for all $y \in \R$. Hence,
\begin{align}\label{chiineq}
f_n  \le \chi_n \le \psi_n \  \text{ for all } n.
\end{align}
By \eqref{maximizing}, this yields the uniform integral bound
\begin{align}\label{intbypart}
\int \chi_{n} \,d(\nu-\mu) \le \nu(\psi_n) - \mu( f_n) \le C \ \text{ for all } l=1,...,L \text{ and } n \in \N.
\end{align}
Using \eqref{intbypart} and the assumption that $(\mu,\nu)$ is irreducible, a local uniform boundedness of $\{\chi_n\}_n$ can be obtained (cf. Step 1 in the proof of \cite[Theorem 1.2]{Lim23}): there exists an increasing sequence of compact intervals $J_{k}:=[c_{k}, d_{k}]$ and constants $M_k \ge 0$ for each $k \in \N$, such that $\cup_{k=1}^\infty J_k = J$, and
\be\label{boundconvex}
0 \le \sup_n \chi_{n} \le M_{k}\, \text{ in }\, J_k.
\ee
{\bf Step 2.} Given any approximating dual optimizer $(\phi_n, \psi_n, \tta_{l,n})$ satisfying \eqref{maximizing}, \eqref{dual}, the goal is to suitably modify it and deduce pointwise convergence of $\phi_n, \psi_n$ to some functions $\phi, \psi$ $\mu, \nu$-a.s. as $n \to \infty$, repectively, where $\phi,\psi \in \R \cup \{+\infty\}$ is $\mu, \nu$-a.s. finite. From convexity of $\chi_n$ with $\mu \preceq_c \nu$, we deduce, for all $n$,
\begin{align}\label{L1bound1}
C \ge \nu(\psi_n) - \mu(f_n) \ge \nu(\chi_n) - \mu(f_n) \ge \mu(\chi_n) - \mu(f_n) = || \chi_n - f_n ||_{L^1(\mu)},
\end{align}
Meanwhile, \eqref{dual} gives  
$f_{n} (x) +  h_{l,n}(x)(y-x) -\psi_n (y) \le -c_l(x,y) \le 0$, hence
\begin{align*}
f_{n} (x) +  h_{l,n}(x)(y-x) -\psi_n (y) \le \chi_n(y) -\psi_n (y) \le 0.
\end{align*}
Integrating by any $\pi \in \cM(\mu,\nu)$ implies
\begin{align}\label{L1bound2}
|| \psi_n - \chi_n ||_{L^1(\nu)} \le \nu(\psi_n) - \mu(f_n) \le C \ \text{ for all } n.
\end{align}
These uniform $L^1$ bounds, combined with the local uniform bound \eqref{boundconvex} and Koml{\'o}s compactness theorem, can imply the desired almost sure convergence of $\{\phi_n\}$ and $\{\psi_n\}$ as presented in \cite{bnt} and in Step 2 in the proof of \cite[Theorem 1.2]{Lim23}, thus we omit the detail here. Also, by following Step 3 in the same proof, one can deduce the following pointwise convergence of $\chi_n$ to a convex function $\chi$
\be
\lim_{n \to \infty} \chi_n(y) = \chi(y) \in \R \, \text{ for every } \, y \in J.
\ee

{\bf Step 3.}  We have obtained the almost sure limit functions $\phi, \psi$, with $f := -\phi$. We may define $\phi := +\infty$ on a $\mu$-null set which includes $\R \setminus I$, and  $\psi := +\infty$ on a $\nu$-null set  which includes $\R \setminus J$, so that they are defined everywhere on $\R$. We will show there exists a function $\tta_l: \R \to \R$, with $h_l := -\tta_l$, $l=1,...,L$, such that 
\begin{align}\label{duallimit}
 \phi (x) + \psi(y) + \tta_{l}(x)(y-x) \ge c_l(x,y).
\end{align}
For any function $f : \R \to \R \cup \{+\infty\}$ which is bounded below by an affine function, let ${\rm conv}[f]:\R \to \R \cup \{+\infty\}$ denote the lower semi-continuous convex envelope of $f$, that is the supremum of all affine functions $\la$ satisfying $ \la \le f$ (If there is no such $\la$, let ${\rm conv}[f] \equiv -\infty$.) Set $H_{l,n}(x, y) := {\rm conv}[ \psi_{n} (\,\cdot\,) - c_l(x,\,\cdot\,) ](y)$. By \eqref{dual},
\begin{align}\label{a1}
f_{n} (x) +  h_{l,n}(x)(y-x) \le H_{l,n}(x, y) \le  \psi_n(y) - c_l(x,y),
\end{align}
because the left hand side is affine in $y$. Letting $y=x$ gives $f_n(x) \le H_{l,n}(x,x)$. 

Next, since the $\limsup$ of convex functions is convex, we have
\begin{align*}
\limsup_{n \to \infty} H_{l,n}(x,y) &\le {\rm conv}[\limsup_{n \to \infty}\big{(}  \psi_{n} (\,\cdot\,) - c_l(x,\,\cdot\,) \big{)}](y)  \\
&\le {\rm conv}[ \psi (\,\cdot\,) - c_l(x,\,\cdot\,)](y) =:H_l(x,y).
\end{align*}
Then by the convergence $f_n \to f$ and the  definition of $H_l(x,y)$, we get
\[
f(x) \le H_l(x,x), \ \text{and} \ H_l(x,y) \le \psi(y) - c_l(x,y).
\]
Set $A := \{ x \in I \, | \, \lim_{n\to\infty} f_n(x) = f(x) \in \R \}$, so that $\mu(A)=1$. Since $y \mapsto H_l(x,y)$ is continuous in $J$ for every $x \in A$ due to the convexity of $y \mapsto H_l(x,y)$ and $\nu$-a.s. finiteness of $\psi$, the subdifferential $\partial H_l(x, \,\cdot\,) (y)$ is nonempty, convex and compact for every $y \in I = {\rm int}(J)$. This allows us to choose a measurable function $h_l : A \to \R$ satisfying $h_l(x) \in \partial H_l(x, \,\cdot\,)(x)$. Such choice  yields \eqref{duallimit} as follows:
 \[
 f(x) + h_l(x) (y-x) \le H_l(x,x) + h_l(x) (y-x) \le H_l(x,y) \le  \psi(y) - c_l(x,y).
  \]
We may define $h_l \equiv 0$ on $\R \setminus A$, noting that $f := -\infty$ on $\R \setminus A$.

{\bf Step 4.} We will show that for any functions $\tta_l : \R \to \R$, $l=1,...,L$ that satisfies  \eqref{duallimit} (whose existence was shown in the previous step), and for any maximizer $\vec \ga^* = (\ga_1^*,...,\ga_L^*) \in \cM_L(\mu, \nu)$ for the problem \eqref{generalrelaxedproblem}, it holds
\begin{align}\label{pointwisedualeq}
 \phi(x) + \psi(y) + \tta_{l}(x)(y-x) = c_l(x,y) \quad \ga_l^* - a.e. \text{ for all } l=1,...,L.
\end{align}
For any $\vec \ga = (\ga_1,...,\ga_L) \in \cM_L(\mu, \nu)$, Assumption {\bf [A]} yields $c_l \in L^1(\ga_l)$. We claim
\begin{align}\label{claim1}
\liminf_{n \to \infty}& \sum_{l=1}^L \int \big( \phi_n(x) + \psi_n(y) + \tta_{l,n}(x)(y-x)  \big)  d\ga_l\\
&\ge  \sum_{l=1}^L \int \big( \phi(x) + \psi(y) + \tta_{l}(x)(y-x)  \big) d\ga_l \ \text{ for every } l. \nn
\end{align}
To see how the claim implies \eqref{pointwisedualeq}, let $\vec \ga^*$ be any maximizer for \eqref{generalrelaxedproblem}. Then
\begin{align*}
\overline P_c&=  \lim_{n \to \infty} \sum_{l=1}^L \int \big( \phi_n(x) + \psi_n(y) + \tta_{l,n}(x)(y-x)  \big)  d\ga_l^*   \\ 
&\ge \sum_{l=1}^L   \liminf_{n \to \infty} \int \big( \phi_n(x) + \psi_n(y) + \tta_{l,n}(x)(y-x)  \big)  d\ga_l^*   \\ 
&\ge  \sum_{l=1}^L \int  \big( \phi(x) + \psi(y) + \tta_{l}(x)(y-x)  \big) d\ga_l^* \\ 
&\ge    \sum_{l=1}^L \int  c_l(x,y) \,d\ga_l^* = \overline P_c,
\end{align*}
hence equality holds throughout. Notice this yields \eqref{pointwisedualeq}, hence the proposition.

To prove \eqref{claim1}, fix any $\vec \ga = (\ga_1,...,\ga_L) \in \cM_L(\mu, \nu)$. The nonnegativity \eqref{dual0} gives $\ga_l^X(\phi_n) + \ga_l^Y(\psi_n) \ge 0$, and \eqref{maximizing} gives $\sum_{l=1}^L ( \ga_l^X(\phi_n) + \ga_l^Y(\psi_n))  = \mu(\phi_n) + \nu(\psi_n) \searrow \overline P_c$. This implies the sequence $\{\ga_l^X(\phi_n) + \ga_l^Y(\psi_n)\}_{n }$ is bounded for all $l$. With this and \eqref{chiineq}, as in Step 2 (but $\ga_l^X \preceq_c \ga_l^Y$ instead of $\mu \preceq_c \nu$), we deduce
\begin{align*}
\sup_n || \chi_{n} +\phi_n ||_{L^1(\ga_l^X)} < \infty, \quad \sup_n ||\psi_n -  \chi_{n}  ||_{L^1(\ga_l^Y)} < \infty, \q \text{for all } l.
\end{align*}
From this, since $\phi_n \to \phi$, $\psi_n \to \psi$, $\chi_n \to \chi$, by Fatou's lemma, we get
\begin{align*}
\chi + \phi \in L^1(\ga_l^X)&, \quad \psi -  \chi  \in L^1(\ga_l^Y),\\
\liminf_{n \to \infty} \int (\chi_n  + \phi_n )\,d\ga_l^X \ge  \int (\chi + \phi )d\ga_l^X&, \ \liminf_{n \to \infty} \int ( \psi_n - \chi_n)\,d\ga_l^Y \ge  \int ( \psi  - \chi )d\ga_l^Y.
\end{align*}
This allows us to proceed
\begin{align*}
&\liminf_{n \to \infty}  \int \big{(} \phi_n(x) + \psi_n(y) + \tta_{l,n}(x)(y-x)  \big{)}  d\ga_l  \\
&= \liminf_{n \to \infty}  \int \big{(} \phi_n(x) + \chi_n(x) - \chi_n(y) + \psi_n(y) -\chi_n(x) + \chi_n(y) +  \tta_{l,n}(x)(y-x)  \big{)} d\ga_l \\
&\ge \int ( \chi + \phi )d\ga_l^X + \int ( \psi - \chi  )d\ga_l^Y + \liminf_{n \to \infty} \int \big{(}\chi_n(y) - \chi_n(x) + \tta_{l,n}(x) (y-x) \big{)}d\ga_l.
\end{align*}
To handle the last term, disintegrate $\ga_l = (\ga_l)_x \otimes \ga_l^X$, and let $\xi_n : I \to \R$ be a sequence of functions satisfying $\xi_n(x) \in \partial \chi_n(x)$. This allows us to proceed
\begin{align*}
&\int \big{(}\chi_n(y) - \chi_n(x) + \tta_{l,n}(x) (y-x) \big{)}d\ga_l  \\
=& \iint \big{(}\chi_n(y) - \chi_n(x) + \tta_{l,n}(x) (y-x) \big{)} (\ga_l)_x (dy) \ga_l^X(dx)   \\
=&\iint  \big{(} \chi_n(y)- \chi_n(x)  +\xi_n(x)(y-x) \big{)}(\ga_l)_x (dy) \ga_l^X(dx),
\end{align*}
because $\int  \tta_{l,n}(x) (y-x) (\ga_l)_x (dy) = \int \xi_n(x) (y-x) (\ga_l)_x (dy) =0$. Notice that the last integrand is nonnegative. Thus by repeated Fatou's lemma, we deduce
\begin{align*}
& \liminf_{n \to \infty} \int \big{(}\chi_n(y) - \chi_n(x) + \tta_{l,n}(x) (y-x) \big{)}d\ga_l    \\
&\ge \int \liminf_{n \to \infty}\bigg{(} \int \big{(} \chi_n(y)- \chi_n(x)  +\xi_n(x)(y-x) \big{)}(\ga_l)_x (dy) \bigg{)} \ga_l^X(dx)  \\
&\ge \int \bigg{(} \int \big{(} \chi(y) - \chi(x) + \xi(x) (y-x) \big{)}(\ga_l)_x (dy) \bigg{)}\ga_l^X(dx), 
\end{align*}
for some $\xi(x)  \in \partial \chi(x)$ which is a limit point of the bounded sequence $\{ \xi_n(x)\}_n$. Finally, in the last line, the inner integral equals
\[
\int \big{(} \chi(y) - \chi(x) + \tta_l(x) (y-x) \big{)} (\ga_l)_x (dy).
\]
This proves the claim, hence the proposition.
\end{proof}

We are prepared to prove Theorem \ref{main}.

\begin{proof}[Proof of Theorem \ref{main}] 
Fix any optimal pair $(\pi, \mu_1)$ for the problem \eqref{problem}, and let $\ga_l = \pi_x \otimes \mu_l$, $l=1,2$, with $\mu_2 = \mu - \mu_1$ and the kernel $\{\pi_x\}_x$ inherited from $\pi$. We understand $c_1(x,y) = c_1(x)$ in the proof. Let us first assume that $\mu \preceq_c \nu$ is irreducible. Because we assume $P_c = \overline P_c$,  by Proposition \ref{dualattainment}, with $f = - \phi$ and $h_l = -\tta_l$, we have 
\begin{align}
&f(x) + h_l(x) (y-x) + c_l(x,y) \le \psi (y) \ \text{ for each $l=1,2$ and } (x,y) \in \R^2, \label{dualineq}\\
&f(x) + h_l(x) (y-x) + c_l(x,y) = \psi (y) \ \ \ga_l - a.e.\, (x,y) \ \text{for each } l=1,2.
\label{dualeq}
\end{align}
Now, saying that an American option holder randomizes her exercise between $c_1,c_2$ is equivalent to saying that the common mass of $\mu_1, \mu_2$ (written as $\mu_1 \wedge \mu_2$) is nonzero. The common mass of $\mu_1, \mu_2$ is defined by the largest measure $\rho = \mu_1 \wedge \mu_2$ satisfying $\rho \le \mu_1$ and $\rho \le \mu_2$. Since $\ga_1$ and $\ga_2$ have the same kernel, \eqref{dualeq} implies
\begin{align}
&f(x) + h_l(x) (y-x) + c_l(x,y) = \psi (y) \q  \pi_x \otimes \rho - a.e.\, (x,y) \text{ for } l=1,2. \label{dualeq2}
\end{align}

Observe that $\psi$ can be taken as $\psi := \max(\psi_1, \psi_2)$, where 
\begin{align}
\psi_l(y) := \sup_x f(x) + h_l(x) (y-x) + c_l(x,y), \nn
\end{align}
and consequently, $\psi_1, \psi_2, \psi$ are all convex since $c_2$ is convex in $y$ (while $c_1$ is independent of $y$.) Now the idea is to differentiate \eqref{dualeq2} by $y$ for $\nu$-a.e.\,$y$, which is enabled by the fact that $\psi$ is differentiable $\nu$-a.s., since $\nu$ is assumed to be absolutely continuous with respect to Lebesgue. By the differentiation combined with the first-order optimality condition from \eqref{dualineq}, \eqref{dualeq} for each $l=1,2$, we deduce 
\begin{align}
 h_1(x)  = \psi' (y) = h_2(x) + (c_2)_y (x,y)  \q \pi_x \otimes \rho - a.e.\, (x,y), \label{dualeq3}
\end{align}
where $ (c_2)_y$ denotes the partial derivative of $c_2$ by $y$, noting that \eqref{dualineq}, \eqref{dualeq} implies $ (c_2)_y (x,y)$ exists $\ga_2$-a.e., since $\psi$ is differentiable $\nu$-a.e..

Now since $c_1 = c_1(x)$, the left hand side of \eqref{dualineq} is linear in $y$ when $l=1$, while $\psi$ is convex. With this, the first equality in \eqref{dualeq3} implies that for $\rho$-a.e.\,$x$, $\psi$ is linear in the smallest interval containing $\spt(\pi_x)$ which contains $x$. Hence,
\begin{align}
\psi'(y) = \psi'(x) \q \pi_x \otimes \rho - a.e.\, (x,y).
\end{align}
The second equality in \eqref{dualeq3} thus becomes
\begin{align}
(c_2)_y(x,y) = \psi'(x) - h_2(x) \q \pi_x \otimes \rho - a.e.\, (x,y). \label{goodeq}
\end{align}
Because $c_2$ is assumed to be strictly convex in $y$, the solution $y$ to \eqref{goodeq} must be unique, and hence, $y=x$ since $\pi_x$ has its barycenter at $x$. We conclude
\begin{align}\label{pidelta}
\pi_x = \delta_x \q \rho-a.e.\,x,
\end{align}
where $\delta_x \in \cP(\R)$ is the Dirac mass at $x$. \eqref{dualeq2} then yields
\begin{align}\label{c1c2equal}
c_1(x) = c_2(x,x) \q \rho-a.e.\,x.
\end{align}
Now if $c_1(x) \ne c_2(x,x) $ $\mu$-a.s., then \eqref{c1c2equal} implies $\rho \equiv 0$, yielding $\mu_1 \perp \mu - \mu_1$ for any optimal pair $(\pi, \mu_1)$. This proves the disjointness when $\mu \preceq_c \nu$ is irreducible. 

%In the general case, observe that \eqref{pidelta}, \eqref{c1c2equal} implies that for $\rho$-a.e.\,$x$, the sure reward $c_1(x)$ and the expected reward $\int c_2(x,y) \pi_x(dy)$ are equal. As a result, the optimal stopping $\mu_1$ is indifferent about the choice between $c_1$ and $c_2$ $\rho$-a.e., implying that one can modify $\mu_1$ to $\tilde \mu_1$ which is also optimal and $\tilde \mu_1 \perp \mu - \tilde \mu_1$. 

%The argument may be stated more rigorously as follows. Let $s_1 = \frac{d\mu_1}{d\mu}$, $s_2= \frac{d\mu_2}{d\mu}$ be the relative density of $\mu_1, \mu_2$ with respect to $\mu$, respectively, which are well defined $\mu$-a.e.\,$x$ with $s_1(x) + s_2 (x) = 1$. Then $s_1 \wedge s_2 = \min (s_1,s_2)$ is a density of $\rho$ with respect to $\mu$. Define, e.g., a modified pair of densities with respect to $\mu$ by %with respect to $\mu$ as
%\be
%\tilde s_1 (x) = \mathds{1}_{\{s_2=0\}}(x) +  \mathds{1}_{\{s_1 \wedge s_2 >0\}}(x), \q \tilde s_2 (x) = \mathds{1}_{\{s_1=0\}}(x).
%\ee
%Define $\tilde \mu_1 = \tilde s_1 \mu$. Then the above argument clearly shows that $(\pi, \tilde \mu_1)$ continues to be optimal, with $\tilde \mu_1 \perp \mu - \tilde \mu_1$. This proves the theorem when $\mu \preceq_c \nu$ is irreducible.

For general $\mu \preceq_c \nu$, it is well known that any convex-ordered pair $(\mu,\nu)$ can be  decomposed as at most countably many irreducible pairs, and the decomposition is uniquely determined by the potential functions $u_\mu, u_\nu$. More precisely, we have:\\
\noindent\cite[Proposition 2.3]{bnt} Let $(I_{k})_{1\leq k \leq N}$ be the open components of the open set $\{u_{\mu}<u_{\nu}\}$ in $\R$, where $N\in \N \cup \{+\infty\}$. Let $I_{0}=\R\setminus \cup_{k\geq1} I_{k}$ and $\mu_{k}=\mu  \big{|}_{_{I_{k}}}$ for $k\geq 0$, so that $\mu =\sum_{k\geq0} \mu_{k}$. There exists a unique decomposition $\nu =\sum_{k\geq0} \nu_{k}$ such that
  \begin{align*}
  \mu_{0} = \nu_{0}, \ \mbox{and} \ (\mu_{k}, \nu_{k}) \mbox{ is irreducible for $k \ge 1$ with } \, \mu_{k}(I_{k})=\mu_{k}(\R).
  \end{align*}
  Moreover, any $\pi \in \cM(\mu,\nu)$ admits a unique decomposition $\pi =\sum_{k\geq0} \pi_{k}$ such that $\pi_{k}\in \cM(\mu_{k},\nu_{k})$ for all $k\geq0$.
  
Here, $\pi_{0}$ must be the identity transport, i.e., $(\pi_{0})_x = \delta_x$,  since it is a martingale transport between the same marginal. Since the theorem has already been proven for the irreducible pairs $(\mu_k, \nu_k)$, $k \ge 1$, we only need to prove it for the identity transport $\pi_0$. In this  case, $\int c_2(x,y) (\pi_0)_x (dy) = c_2(x,x)$, yielding that it is optimal to exercise $c_1$ when $c_1(x) > c_2(x,x)$, while it is optimal to exercise $c_2$ when $c_1(x) < c_2(x,x)$. The assumption $c_1(x) \ne c_2(x,x)$ $\mu$-a.s. therefore proves $\mu_1 \perp \mu - \mu_1$.

Finally, if $(\pi, \mu_1)$ and $(\pi, \mu_1')$ are both optimal, let $\ga_l = \pi_x \otimes \mu_l$ and $\ga_l' = \pi_x \otimes \mu_l'$, $l=1,2$. Let $\tilde \ga_l = (\ga_l + \ga_l')/2$. Then $(\tilde \ga_1, \tilde \ga_2)$ is an optimal solution to \eqref{problem} since $\ga_l$ and $\ga_l'$ share the same kernel. Now $\mu_1 \ne \mu_1'$ implies $\tilde \ga_1^X \not\perp \tilde \ga_2^X$, a contradiction. 
\end{proof}


\begin{thebibliography}{99}
 
 \bibitem{adot19}
A.~Aksamit, S.~Deng, J.~Obłój and X.~Tan. 
\newblock The robust pricing–hedging duality for American options in discrete time financial markets. 
\newblock {\em Mathematical Finance}, Jul;29(3) (2019) 861--897.

\bibitem{bcs18}
E.~Bayraktar, A.M.G.~Cox and Y.~Stoev.
\newblock Martingale Optimal Transport with Stopping.
\newblock {\em SIAM Journal on Control and Optimization}, 56.1 (2018) 417--433.

\bibitem{bhz15}
E.~Bayraktar, Y.J.~Huang and Z.~Zhou.
\newblock On hedging American options under model uncertainty. 
\newblock {\em SIAM J. Financial Math.}, 6 (2015), no. 1, 425--447.

\bibitem{bz16}
E.~Bayraktar and Z.~Zhou.
\newblock Arbitrage, hedging and utility maximization using semi-static trading strategies with American options.
\newblock {\em Ann. Appl. Probab.}, 26 (2016), no. 6, 3531--3558.

\bibitem{bz17}
E.~Bayraktar and Z.~Zhou.
\newblock Super-hedging American Options with Semi-static Trading Strategies under Model Uncertainty.
\newblock {\em Int. J. Theor. Appl. Finance}, 20 (2017) no. 6, 10 pp.

\bibitem{bz19}
E.~Bayraktar and Z.~Zhou.
\newblock No-Arbitrage and Hedging with Liquid American Options. 
\newblock {\em Math. Oper. Res.}, 44 (2019), no. 2, 468--486.

\bibitem{bch}
M.~Beiglb{\"o}ck, A.M.G.~Cox, and M.~Huesmann.
\newblock Optimal transport and Skorokhod embedding.
\newblock {\em Invent. Math.}, 208 (2017) 327--400.

\bibitem{bj}
M.~Beiglb{\"o}ck and N.~Juillet.
\newblock On a problem of optimal transport under marginal martingale constraints.
\newblock {\em Ann. Probab.}, 44(1) (2016) 42--106.

\bibitem{bnt}
M.~Beiglb{\"o}ck, M. Nutz and N. Touzi.
\newblock  Complete duality for martingale optimal transport on the line.
\newblock {\em Ann. Probab.}, 45(5) (2017) 3038--3074. 

\bibitem{CKPS21}
P. Cheridito, M. Kiiski, D.J. Prömel and H.M. Soner.
\newblock Martingale Optimal Transport Duality.
\newblock {\em Mathematische Annalen}, 379 (2021)  1685--1712.

\bibitem{D16}
M.H.A.~Davis.
\newblock Model-Free Methods in Valuation and Hedging of Derivative Securities.
\newblock {\em The Handbook of Post Crisis Financial Modeling}, 2016.

\bibitem{ds1}
Y.~Dolinsky and H.M.~Soner.
\newblock Martingale optimal transport and robust hedging in continuous time.
\newblock {\em Probab. Theory Relat. Fields.}, 160 (2014) 391--427.

\bibitem{GaHeTo11}
A.~Galichon, P.~Henry-Labord{\`e}re, and N.~Touzi.
\newblock A Stochastic Control Approach to No-Arbitrage Bounds Given
  Marginals, with an Application to Lookback Options.
\newblock {\em Annals of Applied Probability.}, Volume 24, Number 1 (2014) 312--336.

\bibitem{H17}
P.~Henry-Labord{\`e}re.
\newblock Model-free Hedging: A Martingale Optimal Transport Viewpoint.
\newblock {\em Chapman \& Hall/CRC Financ. Math. Ser.}
 CRC Press, Boca Raton, FL, 2017.

\bibitem{HT16}
P.~Henry-Labord{\`e}re and N.~Touzi.
\newblock An explicit martingale version of the one-dimensional Brenier theorem. \newblock {\em Finance Stoch.}, 20 (2016) 635--668.

\bibitem{Ho11}
D.~Hobson.
\newblock The {S}korokhod embedding problem and model-independent bounds for option prices.
\newblock In {\em Paris-{P}rinceton {L}ectures on {M}athematical {F}inance
  2010}, volume 2003 of {\em Lecture Notes in Math.}, Springer,
  Berlin (2011)  267--318.

\bibitem{HoNe17}
D.~Hobson and A.~Neuberger.
\newblock Model uncertainty and the pricing of American options. 
\newblock {\em Finance Stoch.}, 21 (2017) 285--329.

\bibitem{HoNo19}
D.~Hobson and D.~Norgilas. 
\newblock Robust bounds for the American put. 
\newblock {\em Finance and Stochastics}, 2019 Apr 15;23(2) (2019) 359--395.

\bibitem{Lim23}
T.~Lim.
\newblock Geometry of vectorial martingale optimal transportations and duality.
\newblock {\em Mathematical Programming}, https://doi.org/10.1007/s10107-023-01954-4 (2023)

\bibitem{N07}
A.~Neuberger.
\newblock Bounds on the American option.
\newblock Preprint (2007). Available online at \url{https://papers.ssrn.com/sol3/papers.cfm?abstract_id=966333}

\bibitem{Obloj}
J.~Ob{\l}{\'o}j.
\newblock The Skorokhod embedding problem and its offspring.
\newblock {\em  Probability Surveys}, 1 (2004) 321--392.

 \bibitem{Sa15}
F. Santambrogio.
\newblock  Optimal transport for applied mathematicians. Calculus of variations, PDEs, and modeling.
\newblock {\em Progress in Nonlinear Differential Equations and their Applications}, 87 Birkhuser/Springer, Cham, (2015)

\end{thebibliography}
\end{document}